\documentclass{amsart}
\usepackage{amsmath,amsthm,amssymb,IMjournal}

\usepackage[a4paper]{hyperref}
\usepackage[noadjust]{cite}

\newcommand{\mymod}[3]{#1 \equiv #2 \kern -0.5em \pmod{#3}}
\newcommand{\mynotmod}[3]{#1 \not \equiv #2 \kern -0.6em \pmod{#3}}

\begin{document}
\newtheorem{theorem}{Theorem}[section]
\newtheorem{corollary}[theorem]{Corollary}
\newtheorem{lemma}[theorem]{Lemma}
\newtheorem{proposition}[theorem]{Proposition}
\theoremstyle{remark}
\newtheorem{remark}[theorem]{Remark}

\theoremstyle{definition}
\newtheorem{definition}[theorem]{Definition}

\numberwithin{equation}{section}

\title{Introduction to generalized Leonardo-Alwyn hybrid numbers}

\author{Gamaliel Cerda-Morales}

%\rec{March 30, 2024}

%\dedicatory{Cordially dedicated to daughter Julieta}\enddedicatory

\abstract 
In \cite{Go}, G\"okba\c{s} defined a new type of number sequence called Leonardo-Alwyn sequence. In this paper, we consider the generalized Leonardo-Alwyn hybrid numbers and investigate some of their properties. We also give some applications related to the generalized Leonardo-Alwyn hybrid numbers in matrices.
\endabstract



\section{Introduction}
Let $p$, $q$ and $r$ be integers such that $p^{2}+4q\neq 0$ the generalized Leonardo-Alwyn sequence $\{L\mathcal{A}_{n}^{(r)}\}_{n\geq 0}$ is described by
\begin{equation}\label{e1}
L\mathcal{A}_{n+3}^{(r)}=(1+p)L\mathcal{A}_{n+2}^{(r)}+(q-p)L\mathcal{A}_{n+1}^{(r)}-qL\mathcal{A}_{n}^{(r)},
\end{equation}
with initial conditions $L\mathcal{A}_{0}^{(r)}=a$, $L\mathcal{A}_{1}^{(r)}=b$ and $L\mathcal{A}_{3}^{(r)}=pb+qa+r$, where $a$ and $b$ are arbitrary real numbers.

The solutions of the equation $x^{3}-(1+p)x^{2}-(q-p)x+q=0$ associated with the recurrence relation (\ref{e1}) are
$$
\varphi=1,\ \psi_{1}=\frac{p+\sqrt{p^{2}+4q}}{2},\ \psi_{2}=\frac{p-\sqrt{p^{2}+4q}}{2}.
$$ 
Note that 
\begin{equation}\label{e2}
\psi_{1}+\psi_{2}=p,\  \psi_{1}\psi_{2}=-q,\ \psi_{1}-\psi_{2}=\sqrt{p^{2}+4q}.
\end{equation}
So the Binet formula for the generalized Leonardo-Alwyn sequence is given by
\begin{equation}\label{bin}
L\mathcal{A}_{n}^{(r)}=\frac{1}{1-p-q}\left[r+\left(\frac{\Phi_{1}\psi_{1}^{n}-\Phi_{2}\psi_{2}^{n}}{\psi_{1}-\psi_{2}}\right)\right],
\end{equation}
where $\Phi_{j}=\Phi(\psi_{j})=((1-p-q)a-r)\psi_{j}+(1-p-q)b+(p^{2}+pq-p)a+(p-1)r$ (with $j=1,2$).

Generalized Leonardo-Alwyn sequence is a generalization of some sequences such as the Leonardo, John-Edouard and Ernst sequences. These sequences have applications in number theory, geometry and algebra. Hence, these sequences have been studied by Alp and Ko\c{c}er \cite{Al,Al1}, Catarino and Borges \cite{Ca1,Ca2}, G\"okba\c{s} \cite{Go}, Shannon \cite{Sh} and Soykan \cite{So1,So2}.

\begin{remark}
If $r=1$, some particular cases of Eq. (\ref{e1}) are: 
\begin{itemize}
\item If $a=b=1$ and $p=q=1$, we get the Leonardo sequence $$Le_{n+2}=Le_{n+1}+Le_{n}+1,\ Le_{0}=Le_{1}=1.$$ Note that $Le_{n}=2F_{n+1}-1$, where $F_{n}$ is the $n$-th Fibonacci number.
\item If $a=b=1$, $p=1$ and $q=2$, we have Ernst sequence $$Er_{n+2}=Er_{n+1}+2Er_{n}+1,\ Er_{0}=Er_{1}=1.$$ Further, we can write $Er_{n}=\frac{1}{2}\left[3J_{n+1}-1\right]$, where $J_{n}$ is the $n$-th Jacobsthal number.
\end{itemize}
\end{remark}

Recently, \"Ozdemir \cite{Oz} defined a new generalization of complex, hyperbolic and dual numbers different from above generalizations. In this generalization, the author gave a system of such numbers that consists of all three number systems together. This set was called hybrid numbers, denoted by $\mathbb{K}$, is defined as
\begin{equation}\label{Hy}
\mathbb{K}=\left\lbrace Z=a+bi+c\varepsilon+dh:\ \begin{array}{c} i^{2}=-1,\ \varepsilon^{2}=0,\ h^{2}=1,\\ 
 ih=-hi=\varepsilon+i \end{array}\right\rbrace,
\end{equation}
where $a$, $b$, $c$ and $d$ are real numbers.

Two hybrid numbers are equal if all their components are equal, one by one. The sum of two hybrid numbers is defined by summing their components. Addition operation in the hybrid numbers is both commutative and associative. Zero is the null element. With respect to the addition operation, the inverse element of $Z$ is $-Z$, which is defined as having all the components of $Z$ changed in their signs. This implies that, $(\mathbb{K}, +)$ is an Abelian group.

The hybridian product is obtained by distributing the terms on the right as in ordinary algebra, preserving that the multiplication order of the units and then writing the values of followings replacing each product of units by the equalities $i^{2}=-1,\ \varepsilon^{2}=0,\ h^{2}=1$ and $ih=-hi=\varepsilon+i$. Using these equalities we can find the product of any two hybrid units. For example, let's find $i\varepsilon$. For this, let's multiply $ih=\varepsilon+i$ by $i$ from the left. Thus, we get $i\varepsilon=1-h$. If we continue in a similar way, we get the following multiplication table.

\begin{table}[ht] 
\caption{The multiplication table for the basis of $\mathbb{K}$.} 
\centering      
\begin{tabular}{lllll}
\hline
$\times $ & $1$ & $i$ & $\varepsilon$ & $h$  \\ \hline
$1$ & $1$ & $i$ & $\varepsilon$ & $h$  
\\ 
$i$ & $i$ & $-1$ & $1-h$ & $\varepsilon+i$
\\ 
$\varepsilon$ & $\varepsilon$ & $1+h$ & $0$ & $-\varepsilon$
\\ 
$h$ & $h$ & $-(\varepsilon+i)$ & $\varepsilon$ & $1$ 
\\ \hline
\end{tabular}
\label{table:1}  
\end{table}
The table \ref{table:1} shows us that the multiplication operation in the hybrid numbers is not commutative. But it has the property of associativity. The conjugate of a hybrid number $Z=a+bi+c\varepsilon+dh$, denoted by $\overline{Z}$, is defined as $\overline{Z}=a-bi-c\varepsilon-dh$ as in the quaternions. The conjugate of the sum of hybrid numbers is equal to the sum of their conjugates. Also, according to the hybridian product, we have $Z\overline{Z}=\overline{Z}Z$. The real number $$\mathcal{C}(Z)=Z\overline{Z}=a^{2}+(b-c)^{2}-c^{2}-d^{2}$$ is called the character of the hybrid number $Z=a+bi+c\varepsilon+dh$. The real number $\sqrt{|\mathcal{C}(Z)|}$ will be called the norm of the hybrid number $Z$ and will be denoted by $\lVert Z\rVert_{\mathbb{K}}$.

Now, we introduce the generalized Leonardo-Alwyn hybrid numbers and investigate some of their properties. We also give some applications related to the generalized Leonardo-Alwyn hybrid numbers in matrices.

\section{Generalized Leonardo-Alwyn hybrid numbers}
For $m\geq 0$, the generalized Leonardo-Alwyn hybrid numbers are defined by
\begin{equation}\label{f1}
La\mathcal{H}_{m}^{(r)}=L\mathcal{A}_{m}^{(r)}+iL\mathcal{A}_{m+1}^{(r)}+\varepsilon L\mathcal{A}_{m+2}^{(r)}+hL\mathcal{A}_{m+3}^{(r)},
\end{equation}
where $L\mathcal{A}_{m}^{(r)}$ is the $m$-th generalized Leonardo-Alwyn number, and $i$, $\varepsilon$, $h$ are hybrid units.

\begin{remark}
For $a=b=1$ and $p=q=r=1$, we obtain $\psi_{1}=\frac{1+\sqrt{5}}{2}$, $\psi_{2}=\frac{1-\sqrt{5}}{2}$, $\Phi_{1}=-2\psi_{1}$ and $\Phi_{2}=-2\psi_{2}$. Then, we obtain the hybrid Leonardo number defined by
$$
HL_{m}^{(1)}=Le_{m}+Le_{m+1}i+Le_{m+2}\varepsilon +Le_{m+3}h.
$$ (see, e.g., \cite{Al1}).
\end{remark}

\begin{theorem}
Let $m\geq 0$ be an integer, Then, we have
\begin{equation}\label{f2}
La\mathcal{H}_{m+2}^{(r)}=pLa\mathcal{H}_{m+1}^{(r)}+qLa\mathcal{H}_{m}^{(r)}+r\Psi,
\end{equation}
with $\Psi=1+i+\varepsilon+h$, $La\mathcal{H}_{0}^{(r)}=a+bi+(pb+qa+r)\varepsilon+((p^{2}+q)b+(pq+q)a+(p+1)r)h$ and $La\mathcal{H}_{1}^{(r)}=b+(pb+qa+r)i+((p^{2}+q)b+(pq+q)a+(p+1)r)\varepsilon+((p^{2}+2pq)b+(p^{2}q+pq+q^{2})a+(p^{2}+p+q+1)r)h$.
\end{theorem}
\begin{proof}
If $m=0$ and $\Psi=1+i+\varepsilon+h$, then we get
\begin{align*}
pLa\mathcal{H}_{1}^{(r)}+qLa\mathcal{H}_{0}^{(r)}+r\Psi&=pb+[p^{2}b+pqa+pr]i\\
&\ \ + [p(p^{2}+q)b+(p^{2}q+pq)a+(p^{2}+p)r]\varepsilon\\
&\ \ + [p^{2}(p+2p)b+p(p^{2}q+pq+q^{2})a+p(p^{2}+p+q+1)r]h\\
&\ \ + qa+qbi+[pqb+q^{2}a+qr]\varepsilon\\
&\ \ + [(p^{2}q+q^{2})b+(pq^{2}+q^{2})a+(pq+q)r]h\\
&\ \ + r+ri+r\varepsilon +rh\\
&=L\mathcal{A}_{2}^{(r)}+iL\mathcal{A}_{3}^{(r)}+\varepsilon L\mathcal{A}_{4}^{(r)}+hL\mathcal{A}_{5}^{(r)}\\
&=La\mathcal{H}_{2}^{(r)}.
\end{align*}
If $m\geq 0$ and using Eqs. (\ref{e1}) and (\ref{f1}), we obtain
\begin{align*}
La\mathcal{H}_{m+2}^{(r)}&=L\mathcal{A}_{m+2}^{(r)}+iL\mathcal{A}_{m+3}^{(r)}+\varepsilon L\mathcal{A}_{m+4}^{(r)}+hL\mathcal{A}_{m+5}^{(r)}\\
&=pL\mathcal{A}_{m+1}^{(r)}+qL\mathcal{A}_{m}^{(r)}+r+ \left[pL\mathcal{A}_{m+2}^{(r)}+qL\mathcal{A}_{m+1}^{(r)}+r\right]i\\
&\ \ + \left[pL\mathcal{A}_{m+3}^{(r)}+qL\mathcal{A}_{m+2}^{(r)}+r\right]\varepsilon+ \left[pL\mathcal{A}_{m+4}^{(r)}+qL\mathcal{A}_{m+3}^{(r)}+r\right]h\\
&=p\left[L\mathcal{A}_{m+1}^{(r)}+iL\mathcal{A}_{m+2}^{(r)}+\varepsilon L\mathcal{A}_{m+3}^{(r)}+hL\mathcal{A}_{m+4}^{(r)}\right]\\
&\ \ + q \left[L\mathcal{A}_{m}^{(r)}+iL\mathcal{A}_{m+1}^{(r)}+\varepsilon L\mathcal{A}_{m+2}^{(r)}+hL\mathcal{A}_{m+3}^{(r)}\right]\\
&\ \ + r+ri+r\varepsilon+rh\\
&=pLa\mathcal{H}_{m+1}^{(r)}+qLa\mathcal{H}_{m}^{(r)}+r\Psi,
\end{align*}
where $\Psi=1+i+\varepsilon+h$ which completes the proof.
\end{proof}

\begin{theorem}[Binet formula of $La\mathcal{H}_{m}^{(r)}$]\label{teo1}
For $m\geq 0$, the Binet formula for generalized Leonardo-Alwyn hybrid numbers is 
\begin{equation}\label{BinH}
La\mathcal{H}_{m}^{(r)}=\frac{1}{1-p-q}\left[r\Psi+\left(\frac{\Phi_{1}\psi_{1}^{m}\Psi_{1}-\Phi_{2}\psi_{2}^{m}\Psi_{2}}{\psi_{1}-\psi_{2}}\right)\right],
\end{equation}
where $\Psi$ as in Eq. (\ref{f2}), $\Psi_{1}=1+\psi_{1}i+\psi_{1}^{2}\varepsilon +\psi_{1}^{3}h$ and $\Psi_{2}=1+\psi_{2}i+\psi_{2}^{2}\varepsilon +\psi_{2}^{3}h$.
\end{theorem}
\begin{proof}
Using Eqs. (\ref{f1}) and (\ref{bin}), we can write the following expression
\begin{align*}
La\mathcal{H}_{m}^{(r)}&=L\mathcal{A}_{m}^{(r)}+iL\mathcal{A}_{m+1}^{(r)}+\varepsilon L\mathcal{A}_{m+2}^{(r)}+hL\mathcal{A}_{m+3}^{(r)}\\
&=\frac{1}{1-p-q}\left[r+\left(\frac{\Phi_{1}\psi_{1}^{m}-\Phi_{2}\psi_{2}^{m}}{\psi_{1}-\psi_{2}}\right)\right]\\
&\ \ + \frac{1}{1-p-q}\left[r+\left(\frac{\Phi_{1}\psi_{1}^{m+1}-\Phi_{2}\psi_{2}^{m+1}}{\psi_{1}-\psi_{2}}\right)\right]i\\
&\ \ + \frac{1}{1-p-q}\left[r+\left(\frac{\Phi_{1}\psi_{1}^{m+2}-\Phi_{2}\psi_{2}^{m+2}}{\psi_{1}-\psi_{2}}\right)\right]\varepsilon \\
&\ \ + \frac{1}{1-p-q}\left[r+\left(\frac{\Phi_{1}\psi_{1}^{m+3}-\Phi_{2}\psi_{2}^{m+3}}{\psi_{1}-\psi_{2}}\right)\right]h\\
&=\frac{1}{1-p-q}\left[r(1+i+\varepsilon+h)+\left(\frac{\Phi_{1}\psi_{1}^{m}\Psi_{1}-\Phi_{2}\psi_{2}^{m}\Psi_{2}}{\psi_{1}-\psi_{2}}\right)\right],
\end{align*}
where $\Psi_{1}=1+\psi_{1}i+\psi_{1}^{2}\varepsilon +\psi_{1}^{3}h$ and $\Psi_{2}=1+\psi_{2}i+\psi_{2}^{2}\varepsilon +\psi_{2}^{3}h$.
\end{proof}

\begin{remark}
For $a=b=1$ and $p=q=r=1$, we obtain $\psi_{1}=\frac{1+\sqrt{5}}{2}$, $\psi_{2}=\frac{1-\sqrt{5}}{2}$, $\Phi_{1}=-2\psi_{1}$ and $\Phi_{2}=-2\psi_{2}$. Then, the Binet formula for the Leonardo hybrid number has the form 
$$
La\mathcal{H}_{m}^{(1)}=2\left(\frac{\psi_{1}^{m+1}\Psi_{1}-\psi_{2}^{m+1}\Psi_{2}}{\psi_{1}-\psi_{2}}\right)-(1+i+\varepsilon+h).
$$
\end{remark}

\begin{proposition}
Let $m\geq 0$ be an integer. Then, the character of the hybrid number $La\mathcal{H}_{m}^{(r)}$ is given by
\begin{equation}
La\mathcal{H}_{m}^{(r)}\overline{La\mathcal{H}_{m}^{(r)}}=\frac{1}{\rho^{2}}\left\lbrace \begin{array}{c} 
2r(1-q-pq)H_{m}-2r(p^{2}+p+q)H_{m+1}\\
+ (1-p^{2}q^{2})H_{m}^{2}+(1-2p-(p^{2}+q)^{2})H_{m+1}^{2}\\
- 2q(1+pq+p^{3})H_{m+1}H_{m}-r^{2}\end{array}\right\rbrace ,
\end{equation}
where $\rho=1-p-q$ and $H_{m}=\frac{\Phi_{1}\psi_{1}^{m}-\Phi_{2}\psi_{2}^{m}}{\psi_{1}-\psi_{2}}$.
\end{proposition}
\begin{proof}
By definition $\mathcal{C}\left(La\mathcal{H}_{m}^{(r)}\right)=La\mathcal{H}_{m}^{(r)}\overline{La\mathcal{H}_{m}^{(r)}}$ and Eq. (\ref{f1}), we have
$$
\mathcal{C}\left(La\mathcal{H}_{m}^{(r)}\right)=\left(L\mathcal{A}_{m}^{(r)}\right)^{2}+\left(L\mathcal{A}_{m+1}^{(r)}-L\mathcal{A}_{m+2}^{(r)}\right)^{2}-\left(L\mathcal{A}_{m+2}^{(r)}\right)^{2}-\left(L\mathcal{A}_{m+3}^{(r)}\right)^{2}.
$$
By Eq. (\ref{e2}), $\rho=1-p-q$ and some elementary calculations we obtain
\begin{align*}
\rho^{2}\mathcal{C}\left(La\mathcal{H}_{m}^{(r)}\right)&=\left[r+\left(\frac{\Phi_{1}\psi_{1}^{m}-\Phi_{2}\psi_{2}^{m}}{\psi_{1}-\psi_{2}}\right)\right]^{2}+ \left[\frac{\Phi_{1}\psi_{1}^{m+1}(1-\psi_{1})-\Phi_{2}\psi_{2}^{m+1}(1-\psi_{2})}{\psi_{1}-\psi_{2}}\right]^{2}\\
&\ \ - \left[r+\left(\frac{\Phi_{1}\psi_{1}^{m+2}-\Phi_{2}\psi_{2}^{m+2}}{\psi_{1}-\psi_{2}}\right)\right]^{2}- \left[r+\left(\frac{\Phi_{1}\psi_{1}^{m+3}-\Phi_{2}\psi_{2}^{m+3}}{\psi_{1}-\psi_{2}}\right)\right]^{2}.
\end{align*}
Then, we have
\begin{align*}
\rho^{2}\mathcal{C}\left(La\mathcal{H}_{m}^{(r)}\right)&=2r(1-q-pq)H_{m}-2r(p^{2}+p+q)H_{m+1}\\
&\ \ + (1-p^{2}q^{2})H_{m}^{2}+(1-2p-(p^{2}+q)^{2})H_{m+1}^{2}\\
&\ \ - 2q(1+pq+p^{3})H_{m+1}H_{m}-r^{2},
\end{align*}
where $H_{m}=\frac{\Phi_{1}\psi_{1}^{m}-\Phi_{2}\psi_{2}^{m}}{\psi_{1}-\psi_{2}}$ and $H_{m+2}=pH_{m+1}+qH_{m}$.
\end{proof}

\begin{theorem}
The generating function for the generalized Leonardo-Alwyn hybrid number sequence $La\mathcal{H}_{m}^{(r)}$ is
\begin{equation}
\sum_{m=0}^{\infty}La\mathcal{H}_{m}^{(r)}t^{m}=\frac{\left\lbrace \begin{array}{c} La\mathcal{H}_{0}^{(r)}+\left(La\mathcal{H}_{1}^{(r)}-(1+p)La\mathcal{H}_{0}^{(r)}\right)t\\ \left(La\mathcal{H}_{2}^{(r)}-(1+p)La\mathcal{H}_{1}^{(r)}-(q-p)La\mathcal{H}_{0}^{(r)}\right)t ^{2}\end{array} \right\rbrace}{1-(1+p)t-(q-p)t^{2}+qt^{3}}.
\end{equation}
\end{theorem}
\begin{proof}
Let $g(t)=\sum_{m=0}^{\infty}La\mathcal{H}_{m}^{(r)}t^{m}$. Then, we obtain
\begin{equation}\label{gen}
g(t)=La\mathcal{H}_{0}^{(r)}+La\mathcal{H}_{1}^{(r)}t+La\mathcal{H}_{2}^{(r)}t^{2}+ \cdots .
\end{equation}
Multiply Eq. (\ref{gen}) on both sides by $-(1+p)t$, $-(q-p)t^{2}$ and $+qt^{3}$. Then, we have
\begin{equation}\label{gen1}
-(1+p)tg(t)=-(1+p)La\mathcal{H}_{0}^{(r)}t-(1+p)La\mathcal{H}_{1}^{(r)}t^{2}-(1+p)La\mathcal{H}_{2}^{(r)}t^{3}+ \cdots .
\end{equation}
\begin{equation}\label{gen2}
-(q-p)t^{2}g(t)=-(q-p)La\mathcal{H}_{0}^{(r)}t^{2}-(q-p)La\mathcal{H}_{1}^{(r)}t^{3}-(q-p)La\mathcal{H}_{2}^{(r)}t^{4}+ \cdots .
\end{equation}
\begin{equation}\label{gen3}
qt^{3}g(t)=qLa\mathcal{H}_{0}^{(r)}t^{3}+qLa\mathcal{H}_{1}^{(r)}t^{4}+qLa\mathcal{H}_{2}^{(r)}t^{5}+ \cdots .
\end{equation}
By adding Eqs. (\ref{gen}), (\ref{gen1}), (\ref{gen2}) and (\ref{gen3}), we have
$$
g(t)=\frac{\left\lbrace \begin{array}{c} La\mathcal{H}_{0}^{(r)}+\left(La\mathcal{H}_{1}^{(r)}-(1+p)La\mathcal{H}_{0}^{(r)}\right)t\\ \left(La\mathcal{H}_{2}^{(r)}-(1+p)La\mathcal{H}_{1}^{(r)}-(q-p)La\mathcal{H}_{0}^{(r)}\right)t ^{2}\end{array} \right\rbrace}{1-(1+p)t-(q-p)t^{2}+qt^{3}}.
$$
\end{proof}

\begin{remark}
For $a=b=1$ and $p=q=r=1$. The generating function for the hybrid Leonardo numbers $HLe_{m}$
$$
g(t)=\frac{1+i+3\varepsilon+5h -(1-i+\varepsilon+h)t+(1-i-\varepsilon-3h)t^{2}}{1-2t+t^{3}}
$$ (see, e.g., \cite{Al1}).
\end{remark}

\begin{theorem}
Let $m\geq 0$ be an integer. Then
$$
\sum_{j=0}^{m}La\mathcal{H}_{j}^{(r)}=\frac{1}{\rho}\left[r\Psi(m+2p+q)\right]+(1-p)La\mathcal{H}_{0}^{(r)}+La\mathcal{H}_{1}^{(r)}-La\mathcal{H}_{m+1}^{(r)}-(p+q)La\mathcal{H}_{m}^{(r)}.
$$
\end{theorem}
\begin{proof}
By using Binet formula of the generalized Leonardo-Alwyn hybrid numbers $La\mathcal{H}_{m}^{(r)}=\frac{1}{\rho}\left[r\Psi+H_{m}\right]$ and $\rho=1-p-q$, we find that
\begin{align*}
\sum_{j=0}^{m}La\mathcal{H}_{j}^{(r)}&=\frac{1}{\rho}\left[r\Psi(m+1)+\sum_{j=0}^{m}H_{j}\right]\\
&=\frac{1}{\rho}\left[r\Psi(m+1)+(1-p)H_{0}+H_{1}-(p+q)H_{m}-H_{m+1}\right]\\
&=\frac{1}{\rho}\left[r\Psi(m+1)+(1-p)H_{0}-(p+q)H_{m}\right]+La\mathcal{H}_{1}^{(r)}-La\mathcal{H}_{m+1}^{(r)}.
\end{align*}
By Eqs. (\ref{e2}), (\ref{f1}) and some elementary calculation, we obtain
\begin{align*}
\sum_{j=0}^{m}La\mathcal{H}_{j}^{(r)}&=\frac{1}{\rho}\left[r\Psi(m+p)-(p+q)H_{m}\right]+(1-p)La\mathcal{H}_{0}^{(r)}+La\mathcal{H}_{1}^{(r)}-La\mathcal{H}_{m+1}^{(r)}\\
&=\frac{1}{\rho}\left[r\Psi(m+2p+q)\right]+(1-p)La\mathcal{H}_{0}^{(r)}+La\mathcal{H}_{1}^{(r)}-La\mathcal{H}_{m+1}^{(r)}-(p+q)La\mathcal{H}_{m}^{(r)}
\end{align*}
as desired.
\end{proof}

\begin{remark}
For $a=b=1$ and $p=q=r=1$. The summation formula of the hybrid Leonardo numbers $HLe_{m}$ is
$$
\sum_{m=0}^{n}HLe_{m}=HLe_{n+2}-(n+2)\Psi -(2i+4\varepsilon+8h)
$$ (see, e.g., \cite{Al1}).
\end{remark}

\begin{theorem}
The exponential generating function for the generalized Leonardo-Alwyn hybrid number $La\mathcal{H}_{m}^{(r)}$ is 
$$
\sum_{m=0}^{\infty}La\mathcal{H}_{m}^{(r)}\frac{t^{m}}{m!}=\frac{1}{1-p-q}\left[r\Psi e^{t}+\left(\frac{\Phi_{1}\Psi_{1}e^{\psi_{1}t}-\Phi_{2}\Psi_{2}e^{\psi_{1}t}}{\psi_{1}-\psi_{2}}\right)\right].
$$
\end{theorem}
\begin{proof}
By considering Theorem \ref{teo1} and $\rho=1-p-q$, we have
\begin{align*}
\sum_{m=0}^{\infty}La\mathcal{H}_{m}^{(r)}\frac{t^{m}}{m!}&=\frac{1}{\rho}\left[r\Psi\sum_{m=0}^{\infty}\frac{t^{m}}{m!}+\left(\frac{\Phi_{1}\Psi_{1}\sum_{m=0}^{\infty}\frac{(\psi_{1}t)^{m}}{m!}-\Phi_{2}\Psi_{2}\sum_{m=0}^{\infty}\frac{(\psi_{2}t)^{m}}{m!}}{\psi_{1}-\psi_{2}}\right)\right]\\
&=\frac{1}{\rho}\left[r\Psi e^{t}+\left(\frac{\Phi_{1}\Psi_{1}e^{\psi_{1}t}-\Phi_{2}\Psi_{2}e^{\psi_{1}t}}{\psi_{1}-\psi_{2}}\right)\right].
\end{align*}
\end{proof}

For simplicity of notation, let
\begin{equation}\label{sam}
\mathcal{H}_{n}=\frac{\Phi_{1}\psi_{1}^{m}\Psi_{1}-\Phi_{2}\psi_{2}^{m}\Psi_{2}}{\psi_{1}-\psi_{2}}.
\end{equation}
Then, if $\rho=1-p-q$, the Binet formula of the generalized Leonardo-Alwyn hybrid numbers is given by
\begin{equation}\label{sim}
La\mathcal{H}_{n}^{(r)}=\frac{1}{\rho}\left[r\Psi+\mathcal{H}_{n}\right].
\end{equation}
Note that $\mathcal{H}_{n+2}=p\mathcal{H}_{n+1}+q\mathcal{H}_{n}$ and $\psi_{j}^{2}=p\psi_{j}+q$ for $j=1,2$.

The Vajda's identity for the sequence $\mathcal{H}_{n}$ and generalized Leonardo-Alwyn hybrid number is given in the next theorem.
\begin{theorem}\label{pp}
Let $n\geq 0$,  $u\geq 0$, $v\geq 0$ be integers. Then, we have
\begin{equation}\label{t1}
\mathcal{H}_{n+u}\mathcal{H}_{n+v}-\mathcal{H}_{n}\mathcal{H}_{n+u+v}=\frac{1}{\Delta^{2}}\left(\Phi_{1}\Phi_{2}(-q)^{n}(\psi_{1}^{u}-\psi_{2}^{u})\left[\psi_{1}^{v}\Psi_{2}\Psi_{1}-\psi_{2}^{v}\Psi_{1}\Psi_{2}\right]\right),
\end{equation}
\begin{equation}\label{t2}
\begin{aligned}
La\mathcal{H}_{n+u}^{(r)}La\mathcal{H}_{n+v}^{(r)}&-La\mathcal{H}_{n}^{(r)}La\mathcal{H}_{n+u+v}^{(r)}\\
&=\frac{1}{\rho^{2}}\left\lbrace \begin{array}{cc} 
\Phi_{1}\Phi_{2}(-q)^{n}(\psi_{1}^{u}-\psi_{2}^{u})\left[\psi_{1}^{v}\Psi_{2}\Psi_{1}-\psi_{2}^{v}\Psi_{1}\Psi_{2}\right]\\
+ r\left[\Psi\mathcal{K}_{n}(u)-\mathcal{K}_{n+v}(u)\Psi\right]
\end{array} \right\rbrace,
\end{aligned}
\end{equation}
where $\Delta=\psi_{1}-\psi_{2}$ and $\mathcal{K}_{n}(u)=\mathcal{H}_{n}-\mathcal{H}_{n+u}$.
\end{theorem}
\begin{proof}
(\ref{t1}): Using Eq. (\ref{sam}), $\Phi_{1}$, $\Phi_{2}$, $\Psi_{1}$ and $\Psi_{2}$ as in Theorem \ref{BinH}, we have
\begin{align*}
(\psi_{1}-\psi_{2})^{2}&\left[\mathcal{H}_{n+u}\mathcal{H}_{n+v}-\mathcal{H}_{n}\mathcal{H}_{n+u+v}\right]\\
&=\left(\Phi_{1}\psi_{1}^{n+u}\Psi_{1}-\Phi_{2}\psi_{2}^{n+u}\Psi_{2}\right)\left(\Phi_{1}\psi_{1}^{n+v}\Psi_{1}-\Phi_{2}\psi_{2}^{n+v}\Psi_{2}\right)\\
&\ \ - \left(\Phi_{1}\psi_{1}^{n}\Psi_{1}-\Phi_{2}\psi_{2}^{n}\Psi_{2}\right)\left(\Phi_{1}\psi_{1}^{n+u+v}\Psi_{1}-\Phi_{2}\psi_{2}^{n+u+v}\Psi_{2}\right)\\
&=\Phi_{1}\Phi_{2}(-q)^{n}(\psi_{1}^{u}-\psi_{2}^{u})\left[\psi_{1}^{v}\Psi_{2}\Psi_{1}-\psi_{2}^{v}\Psi_{1}\Psi_{2}\right].
\end{align*}
Note that $\Psi_{1}\Psi_{2}\neq \Psi_{2}\Psi_{1}$ in above equation.

By formula (\ref{sim}) and Eq. (\ref{t1}), we get
\begin{align*}
\rho^{2}&\left[La\mathcal{H}_{n+u}^{(r)}La\mathcal{H}_{n+v}^{(r)}-La\mathcal{H}_{n}^{(r)}La\mathcal{H}_{n+u+v}^{(r)}\right]\\
&=\left[r\Psi+\mathcal{H}_{n+u}\right]\left[r\Psi+\mathcal{H}_{n+v}\right]- \left[r\Psi+\mathcal{H}_{n}\right]\left[r\Psi+\mathcal{H}_{n+u+v}\right]\\
&=\mathcal{H}_{n+u}\mathcal{H}_{n+v}-\mathcal{H}_{n}\mathcal{H}_{n+u+v}+ r\left[\Psi\mathcal{K}_{n}(u)-\mathcal{K}_{n+v}(u)\Psi\right]\\
&=\frac{1}{\Delta^{2}}\left(\Phi_{1}\Phi_{2}(-q)^{n}(\psi_{1}^{u}-\psi_{2}^{u})\left[\psi_{1}^{v}\Psi_{2}\Psi_{1}-\psi_{2}^{v}\Psi_{1}\Psi_{2}\right]\right)+ r\left[\Psi\mathcal{K}_{n}(u)-\mathcal{K}_{n+v}(u)\Psi\right],
\end{align*}
where $\mathcal{K}_{n}(u)=\mathcal{H}_{n}-\mathcal{H}_{n+u}$.
\end{proof}

It is easily seen that for special values of $u$ and $v$ by Theorem \ref{pp}, we get new identities for generalized Leonardo-Alwyn hybrid numbers:
\begin{itemize}
\item Catalan's identity: $v=-u$.
\item Cassini's identity: $u=1$, $v=-1$.
\item d'Ocagne's identity: $u=1$, $v=m-n$, with $m\geq n$.
\end{itemize}

\begin{corollary}
Catalan identity for generalized Leonardo-Alwyn hybrid numbers. Let $n\geq 0$, $p\geq 0$ be integers such that $n\geq p$. Then
\begin{equation}\label{c1}
\begin{aligned}
La\mathcal{H}_{n+u}^{(r)}La\mathcal{H}_{n-u}^{(r)}&-\left(La\mathcal{H}_{n}^{(r)}\right)^{2}\\
&=\frac{1}{\rho^{2}}\left\lbrace \begin{array}{cc} 
\frac{1}{\Delta^{2}}\Phi_{1}\Phi_{2}(-q)^{n-u}(\psi_{1}^{u}-\psi_{2}^{u})\left[\psi_{2}^{u}\Psi_{2}\Psi_{1}-\psi_{1}^{u}\Psi_{1}\Psi_{2}\right]\\
+ r\left[\Psi\mathcal{K}_{n}(u)-\mathcal{K}_{n-u}(u)\Psi\right]
\end{array} \right\rbrace .
\end{aligned}
\end{equation}
\end{corollary}

\begin{corollary}
Cassini identity for generalized Leonardo-Alwyn hybrid numbers. Let $n\geq 1$ be an integer. Then
\begin{equation}\label{c2}
\begin{aligned}
La\mathcal{H}_{n+1}^{(r)}La\mathcal{H}_{n-1}^{(r)}&-\left(La\mathcal{H}_{n}^{(r)}\right)^{2}\\
&=\frac{1}{\rho^{2}}\left\lbrace \begin{array}{cc} 
\frac{1}{\Delta^{2}}\Phi_{1}\Phi_{2}(-q)^{n-1}(\psi_{1}-\psi_{2})\left[\psi_{2}\Psi_{2}\Psi_{1}-\psi_{1}\Psi_{1}\Psi_{2}\right]\\
+ r\left[\Psi\mathcal{K}_{n}(1)-\mathcal{K}_{n-1}(1)\Psi\right]
\end{array} \right\rbrace .
\end{aligned}
\end{equation}
\end{corollary}

\begin{corollary}
d'Ocagne identity for generalized Leonardo-Alwyn hybrid numbers. Let $n\geq 0$, $m\geq 0$ be integers such that $m\geq n$. Then
\begin{equation}\label{c3}
\begin{aligned}
La\mathcal{H}_{n+1}^{(r)}La\mathcal{H}_{m}^{(r)}&-La\mathcal{H}_{n}^{(r)}La\mathcal{H}_{m+1}^{(r)}\\
&=\frac{1}{\rho^{2}}\left\lbrace \begin{array}{cc} 
\frac{1}{\Delta^{2}}\Phi_{1}\Phi_{2}(-q)^{n}(\psi_{1}-\psi_{2})\left[\psi_{1}^{m-n}\Psi_{2}\Psi_{1}-\psi_{2}^{m-n}\Psi_{1}\Psi_{2}\right]\\
+ r\left[\Psi\mathcal{K}_{n}(1)-\mathcal{K}_{m}(1)\Psi\right]
\end{array} \right\rbrace .
\end{aligned}
\end{equation}
\end{corollary}

\section{Matrix representation of generalized Leonardo-Alwyn hybrid numbers}
Now, we will give the matrix representation of generalized Leonardo-Alwyn hybrid numbers. Also, we obtain a formula for generalized Leonardo-Alwyn hybrid number $La\mathcal{H}_{m}^{(r)}$, in terms of tridiagonal determinant, by using the same kind of approach that was used by Cereceda in \cite{Ce}.

\begin{theorem}
Let $m\geq 0$ be an integer. Then
\begin{equation}
\left[
\begin{array}{c}
La\mathcal{H}_{m+3}^{(r)} \\ 
La\mathcal{H}_{m+2}^{(r)} \\ 
La\mathcal{H}_{m+1}^{(r)}
\end{array}
\right]=\textbf{Q}_{\mathcal{H}} \cdot\left[
\begin{array}{c}
La\mathcal{H}_{m+2}^{(r)} \\ 
La\mathcal{H}_{m+1}^{(r)} \\ 
La\mathcal{H}_{m}^{(r)}
\end{array}
\right],
\end{equation}
where $$\textbf{Q}_{\mathcal{H}}=\left[
\begin{array}{ccc}
1+p & q-p& -q \\ 
1& 0 & 0 \\ 
0 & 1& 0
\end{array}
\right].$$
\end{theorem}

\begin{theorem}
Let $m\geq 0$ be an integer. Then
\begin{equation}\label{mat}
\left[
\begin{array}{ccc}
La\mathcal{H}_{m+3}^{(r)} & La\mathcal{G}_{m+4}^{(r)}& -qLa\mathcal{H}_{m+2}^{(r)}\\ 
La\mathcal{H}_{m+2}^{(r)}& La\mathcal{G}_{m+3}^{(r)}& -qLa\mathcal{H}_{m+1}^{(r)} \\
La\mathcal{H}_{m+1}^{(r)}& La\mathcal{G}_{m+2}^{(r)}& -qLa\mathcal{H}_{m}^{(r)}
\end{array}
\right]=\left[
\begin{array}{ccc}
La\mathcal{H}_{3}^{(r)} & La\mathcal{G}_{4}^{(r)}& -qLa\mathcal{H}_{2}^{(r)}\\ 
La\mathcal{H}_{2}^{(r)}& La\mathcal{G}_{3}^{(r)}& -qLa\mathcal{H}_{1}^{(r)} \\
La\mathcal{H}_{1}^{(r)}& La\mathcal{G}_{2}^{(r)}& -qLa\mathcal{H}_{0}^{(r)}
\end{array}
\right]\cdot \textbf{Q}_{\mathcal{H}}^{m},
\end{equation}
where $La\mathcal{G}_{m+1}^{(r)}=La\mathcal{H}_{m+1}^{(r)}-(1+p)La\mathcal{H}_{m}^{(r)}$.
\end{theorem}
\begin{proof}
We use induction on $m$. If $m=0$ then the result is obvious. Assuming the formula (\ref{mat}) holds for $m\geq 0$, we shall prove it for $m+1$. Using the induction's hypothesis and formula (\ref{f1}), we have
\begin{align*}
&\left[
\begin{array}{ccc}
La\mathcal{H}_{3}^{(r)} & La\mathcal{G}_{4}^{(r)}& -qLa\mathcal{H}_{2}^{(r)}\\ 
La\mathcal{H}_{2}^{(r)}& La\mathcal{G}_{3}^{(r)}& -qLa\mathcal{H}_{1}^{(r)} \\
La\mathcal{H}_{1}^{(r)}& La\mathcal{G}_{2}^{(r)}& -qLa\mathcal{H}_{0}^{(r)}
\end{array}
\right]\cdot \textbf{Q}_{\mathcal{H}}^{m+1}\\
&=\left[
\begin{array}{ccc}
La\mathcal{H}_{3}^{(r)} & La\mathcal{G}_{4}^{(r)}& -qLa\mathcal{H}_{2}^{(r)}\\ 
La\mathcal{H}_{2}^{(r)}& La\mathcal{G}_{3}^{(r)}& -qLa\mathcal{H}_{1}^{(r)} \\
La\mathcal{H}_{1}^{(r)}& La\mathcal{G}_{2}^{(r)}& -qLa\mathcal{H}_{0}^{(r)}
\end{array}
\right]\cdot \textbf{Q}_{\mathcal{H}}^{m}\cdot \left[
\begin{array}{ccc}
1+p& q-p& -q \\ 
1& 0 & 0 \\ 
0 & 1& 0
\end{array}
\right]\\
&=\left[
\begin{array}{ccc}
La\mathcal{H}_{m+3}^{(r)} & La\mathcal{G}_{m+4}^{(r)}& -qLa\mathcal{H}_{m+2}^{(r)}\\ 
La\mathcal{H}_{m+2}^{(r)}& La\mathcal{G}_{m+3}^{(r)}& -qLa\mathcal{H}_{m+1}^{(r)} \\
La\mathcal{H}_{m+1}^{(r)}& La\mathcal{G}_{m+2}^{(r)}& -qLa\mathcal{H}_{m}^{(r)}
\end{array}
\right]\left[
\begin{array}{ccc}
1+p& q-p& -p \\ 
1& 0 & 0 \\ 
0 & 1& 0
\end{array}
\right]\\
&=\left[
\begin{array}{ccc}
La\mathcal{H}_{m+4}^{(r)} & (q-p)La\mathcal{G}_{m+3}^{(r)}-qLa\mathcal{G}_{m+2}^{(r)}& -qLa\mathcal{H}_{m+3}^{(r)}\\ 
La\mathcal{H}_{m+3}^{(r)}& (q-p)La\mathcal{G}_{m+2}^{(r)}-qLa\mathcal{G}_{m+1}^{(r)}& -qLa\mathcal{H}_{m+2}^{(r)} \\
La\mathcal{H}_{m+2}^{(r)}& (q-p)La\mathcal{G}_{m+1}^{(r)}-qLa\mathcal{G}_{m}^{(r)}& -qLa\mathcal{H}_{m+1}^{(r)}
\end{array}
\right]\\
&=\left[
\begin{array}{ccc}
La\mathcal{H}_{m+4}^{(r)} & La\mathcal{G}_{m+5}^{(r)}& -qLa\mathcal{H}_{m+3}^{(r)}\\ 
La\mathcal{H}_{m+3}^{(r)}& La\mathcal{G}_{m+4}^{(r)}& -qLa\mathcal{H}_{m+2}^{(r)} \\
La\mathcal{H}_{m+2}^{(r)}& La\mathcal{G}_{m+3}^{(r)}& -qLa\mathcal{H}_{m+1}^{(r)}
\end{array}
\right]
\end{align*}
which ends the proof.
\end{proof}

\begin{remark}
For $a=b=1$ and $p=q=r=1$, we have the matrix representation of hybrid Leonardo numbers $HLe_{m}$ as follows
$$
\left[
\begin{array}{ccc}
HLe_{m+3} & -HLe_{m+1} & -HLe_{m+2} \\ 
HLe_{m+2} & -HLe_{m} & -HLe_{m+1} \\ 
HLe_{m+1} & -HLe_{m-1} & -HLe_{m} \\ 
\end{array}
\right]=\left[
\begin{array}{ccc}
HLe_{3} & -HLe_{1} & -HLe_{2} \\ 
HLe_{2} & -HLe_{0} & -HLe_{1} \\ 
HLe_{1} & -HLe_{-1} & -HLe_{0} \\ 
\end{array}
\right]\cdot \textbf{Q}_{\mathcal{H}}^{m},
$$
where $$\textbf{Q}_{\mathcal{H}}=\left[
\begin{array}{ccc}
2 & 0& -1 \\ 
1& 0 & 0 \\ 
0 & 1& 0
\end{array}
\right].$$
\end{remark}

Now, we present another way to obtain the $n$-th term of the generalized Leonardo-Alwyn hybrid sequence as the computation of a tridiagonal matrix. We shall adopt the ideas stated in \cite[p. 277]{Ce} using the following result which is stated in such work:
\begin{theorem}\label{cer}
Let $\{x_{n}\}_{n\geq1}$ be any third-order linear sequence defined recursively by the following: $$x_{n+3}=ux_{n+2}+vx_{n+1}+wx_{n},\  n\geq 0,$$ with $x_{0}=A\neq 0$, $x_{1}=B$ and $x_{2}=C$. Then, for all $n\geq 0$:
\begin{equation} \label{f1}
x_{n}= \left| \begin{matrix}
A  & 1     & 0      &   \cdots & \cdots     & \cdots     & 0  \\
Au-B  & u & \frac{1}{A} & 0 &\cdots &\cdots        & 0 \\
0 & Bu-C & u & w & \ddots & \cdots & 0    \\
0  &A & -\frac{v}{w}  & u & w& \ddots & 0  \\
0  &0 & \frac{1}{w}  & -\frac{v}{w} & u & &   \\
\vdots &  \vdots &     & \ddots & \ddots & \ddots & w \\
0 & \cdots    & &     0   & \frac{1}{w}     & -\frac{v}{w}     & u\end{matrix} \right|_{(n+1)\times (n+1)}
\end{equation}
\end{theorem}

In the case of the generalized Leonardo-Alwyn sequence and taking into account the recurrence relation (\ref{f2}), we have $u=1+p$, $v=q-p$, $w=-q$, and initial conditions $La\mathcal{H}_{0}^{(r)}$, $La\mathcal{H}_{1}^{(r)}$ and $La\mathcal{H}_{2}^{(r)}$.Then, by the use of the previous theorem, we have the following result which gives a different way to calculate the $n$-th term of this sequence:
\begin{proposition}
For $n\geq 0$, we have 
\begin{equation} \label{f2}
La\mathcal{H}_{n}^{(r)}= \left| \begin{matrix}
La\mathcal{H}_{0}^{(r)}  & 1     & 0      &   \cdots & \cdots     & \cdots     & 0  \\
Au-B & 1+p & \left(La\mathcal{H}_{0}^{(r)}\right)^{-1} & 0 &\cdots &\cdots        & 0 \\
0 & Bu-C& 1+p & -q & \ddots & \cdots & 0    \\
0  &La\mathcal{H}_{0}^{(r)} & \frac{q-p}{q}  & 1+p & -q & \ddots & 0  \\
0  &0 & \frac{1}{2}  & \frac{q-p}{q} & 1+p & &   \\
\vdots &  \vdots &     & \ddots & \ddots & \ddots & -q \\
0 & \cdots    & &     0   & -\frac{1}{q}     & \frac{q-p}{q}     & 1+p\end{matrix} \right|_{(n+1)\times (n+1)},
\end{equation}
where $Au-B=(1+p)La\mathcal{H}_{0}^{(r)}-La\mathcal{H}_{1}^{(r)}$ and $Bu-C=(1+p)La\mathcal{H}_{1}^{(r)}-La\mathcal{H}_{2}^{(r)}$.
\end{proposition}

\section{Conclusions}
In this paper, the sequence of generalized Leonardo-Alwyn hybrid numbers defined by a recurrence relation of third order was introduced. Some properties involving this sequence, including the Binet formula and the generating function, were presented. In the future, we intend to discuss the generalized Leonardo-Alwyn hybrid number with negative subscripts and the generalized Leonardo-Alwyn quaternion.

%%%%%%%%%%%%%%%%

%\section{Acknowledgement}
%The author would like to thank the anonymous referee for his/her valuable comments and suggestions.

{\em Gamaliel Cerda-Morales},\\
Instituto de Matem\'aticas, Pontificia Universidad Cat\'olica de Valpara\'iso,\\
Blanco Viel 596, Valpara\'iso, Chile. \\
E-mail: \texttt{gamaliel.cerda.m@mail.pucv.cl}.


\begin{thebibliography}{999}
\bibitem{Al}
Y. Alp, E.G. Ko\c{c}er, Some properties of Leonardo numbers. Konuralp J Math \textbf{9}(1) (2023), 183--189.
\bibitem{Al1}
Y. Alp, E.G. Ko\c{c}er, Hybrid Leonardo numbers, Chaos, Solitons and Fractals \textbf{150} (2021), 111128
\bibitem{Ca1}
P. Catarino, A. Borges, On Leonardo numbers. Acta Mathematica Universitatis Comenianae \textbf{89}(1) (2019), 75--86.
\bibitem{Ca2}
 P. Catarino, A. Borges, A Note on Incomplete Leonardo Numbers, Integers, Vol:20 (2020).
 \bibitem{Ce}
 J.L. Cereceda, Determinantal Representations for Generalized Fibonacci and Tribonacci Numbers, Int. J. Contemp. Math. Sciences \textbf{9}(6) (2014), 269--285.
\bibitem{Go} 
H. G\"okba\c{s}, A New Family of Number Sequences: Leonardo-Alwyn Numbers, Armenian Journal of Mathematics \textbf{15}(6) (2023), 1--13.\\
\url{https://doi.org/10.52737/18291163-2023.15.6-1-13}
\bibitem{Oz} 
M. \"Ozdemir, Introduction to hybrid numbers, Adv. Appl. Clifford Algebras \textbf{28}:11 (2018).
\bibitem{Sh} 
 G. Shannon, A Note On Generalized Leonardo Numbers, Notes on Number Theory and Discrete Mathematics, \textbf{25}(3) (2019), 97--101.
\bibitem{So1} 
Y. Soykan, Generalized Edouard Numbers, Int. J. Adv. Appl. Math. and Mech. \textbf{9}(3) (2022), 41--52. 
\bibitem{So2} 
Y. Soykan, Generalized Ernst Numbers, Asian Journal of Pure and Applied Mathematics \textbf{4}(3) (2022), 1--15. 
\end{thebibliography}
\end{document}